\documentclass[12pt]{amsart}
\usepackage{amssymb}
\usepackage{graphicx}
\usepackage{amsfonts}
\usepackage{amsmath}
\usepackage{amsthm}
\usepackage{fancyhdr}
\usepackage{indentfirst}
\usepackage{hyperref}
\usepackage{comment}
\usepackage{color}
\usepackage{subcaption}
\usepackage{epsfig}
\usepackage{pst-grad} 
\usepackage{pst-plot} 
\usepackage[space]{grffile} 
\usepackage{etoolbox} 
\usepackage{mathrsfs} 
\usepackage{enumitem} 
\usepackage[margin=1.5in]{geometry}
\usepackage[normalem]{ulem}

\newtheorem{thm}{Theorem}[section]
\newtheorem{remark}[thm]{Remark}
\newtheorem*{proposition*}{Proposition}
\newtheorem{proposition}[thm]{Proposition}
\newtheorem{conjecture}[thm]{Conjecture}
\newtheorem*{defn*}{Definition}
\newtheorem{defn}{Definition}
\newtheorem{prop}[thm]{Proposition}

\newtheorem{lemma}[thm]{Lemma}
\newtheorem{theorem}[thm]{Theorem}

\newtheorem{cor}[thm]{Corollary}
\newtheorem{definition}[thm]{Definition}
\newtheorem{ex}[thm]{Example}

\def\two{\uppercase\expandafter{\romannumeral2}}

\thicklines
\usepackage[perpage,symbol*]{footmisc}
\DefineFNsymbols{circled}{{\ding{192}}{\ding{193}}{\ding{194}}
{\ding{195}}{\ding{196}}{\ding{197}}{\ding{198}}{\ding{199}}{\ding{200}}{\ding{201}}}
\setfnsymbol{circled}

\def\->{\rightarrow}
\def\=>{\xrightarrow}

\title[Yamabe invariants on compact manifolds]{On the Yamabe invariant of certain compact manifolds with boundary}
\author{Xuan Yao}
\address{Department of Mathematics, Malott Hall, Cornell University, Ithaca, NY 14853, USA}
\email{xy346@cornell.edu}
\date{Nov 2022}
\begin{document}

\begin{abstract}
We generalize Kobayashi's connected-sum inequality to the $\lambda$-Yamabe invariants. As an application, we calculate 
 the $\lambda$-Yamabe invariants of $\#m_1\mathbb{RP}^n\# m_2(\mathbb{RP}^{n-1}\times S^1)\#lH^n\#kS_+^n$, for any $\lambda\in [0,1]$,  $n\geq 3$, provided $k+l\geq 1$.

As a corollary, we prove that $\mathbb{RP}^n$ minus finitely many disjoint $n$-balls have the same $\lambda$-Yamabe invariants as the hemi-sphere, which forms an interesting contrast with the famous Bray-Neves results \cite{bray2004classification} on the Yamabe invariants of $\mathbb{RP}^3$.

\end{abstract}

\maketitle
\section{Introduction}
We begin by reminding the reader of the definitions of generalized Yamabe constants and Yamabe invariant of closed $n$-manifolds. 

Given a closed $n$-dimensional Riemannian manifold $(M,g)$, the Einstein-Hilbert energy functional is defined as
\[
E(g)=\frac{\int_{M}R_gdV_g}{(\int_MdV_g)^{\frac{n-2}{n}}},
\]
where $R_g$ is the scalar curvature.

Since $E$ is unbounded in either positive nor negative direction, Yamabe \cite{yamabe1960deformation} proposed to study the minimal value of this energy functional in a conformal class of metrics. We use $[g]$ to denote the conformal class of metrics on $M$ which contains $g$,  that is,
\[
[g]=\{u^{\frac{4}{n-2}}g: u\in H^1(M,g), u>0\}.
\]
The minimal value is called the {\em Yamabe constant} of this conformal class, denoted by $Y(M, [g])$. More precisely, 
\[
Y(M,[g])=\inf_{u\in H^1(M,g)}\frac{\int_{M}\frac{4(n-1)}{n-2}|\nabla u|_g^2dV_g+\int_{M}R_gdV_g}{(\int_{M}|u|^{\frac{2n}{n-2}}dV_g)^{\frac{n-2}{n}}}.
\]
The minimizer of the Hilbert-Einstein functional in each conformal class is a metric with constant scalar curvature. The existence of constant scalar curvature metrics in each conformal is known as the Yamabe Problem \cite{yamabe1960deformation}. This problem was solved after decades of efforts by Yamabe \cite{yamabe1960deformation}, Trudinger \cite{trudinger1968remarks}, Aubin \cite{aubin1976equations}, and finally completed by Schoen \cite{schoen1984conformal}.

By taking the supremum of Yamabe constants over all conformal classes of metrics on $M$, we could obtain a smooth topological invariant which is called the {\em Yamabe invariant}
\[
\sigma(M)=\sup_{[g]\in\mathcal C}Y(M,[g]),
\] 
where $\mathcal C$ denotes the set of all the conformal classes of metrics on $M$.

It is natural to generalize the Yamabe problem to compact manifolds with boundary. Given any $(M,g)$, a compact manifold with boundary, does there exist a metric $\tilde{g}$ conformal to $g$, such that $(M,\tilde{g})$ has constant scalar curvature on $M$, and constant mean curvature on $\partial M$? 

Escobar \cite{Escobar96} showed that for almost any Riemannian $(M,g)$, there exists a metric within the conformal class of $g$ having constant scalar curvature on $M$ and constant mean curvature on $\partial M$. To prove the generalized Yamabe Problem, Escobar \cite{Escobar96} defined the generalized Yamabe constants for manifolds with boundary. Based on his definitions, we define the $\lambda$-Yamabe constants and the $\lambda$-Yamabe invariants for compact manifolds with boundary; see Definition \ref{lambdayamabe}.

Akutagawa and Botvinnik \cite{akutagawa2018relative} defined the Relative Yamabe invariant to describe the Yamabe invariant for the manifold with boundary. They proved the approxiamation theorem, gluing theorem and some useful inequalities for Manifold with non-positive Yamabe invariant.

Now we state the Main result of this paper

\begin{theorem}\label{maintheorem}
Suppose $m_1,m_2,k,l$ are all non-negative integers, and $k+l\geq 1$, then
\begin{equation}\label{Mainequ}
\sigma_{\lambda}(\#m_1\mathbb{RP}^n\#m_2(\mathbb{RP}^{n-1}\times S^1)\#l H^n\#k S^n_+)=\sigma_{\lambda}(S^n_+),\quad \forall \lambda\in [0,1].
\end{equation}
Here $\mathbb{RP}^n$, $H^n$, $S^n_+$ denote $n$-dimensional projective planes, $n$-dimensional handle bodies and $n$-dimensional hemi-spheres respectively. $\sigma_\lambda$ denotes the $\lambda$-Yamabe invariant; see Definition \ref{lambdayamabe}.
\end{theorem}

The proof of the Main Theorem relies on the generalization of Kobayashi's connected-sum inequality \cite{Kobayashi1987/88}. This inequality \cite{Kobayashi1987/88} is an important tool for estimating Yamabe invariants of closed manifolds; it states the fact that the Yamabe invariant of the connected sum of two closed manifolds is greater than or equal to the Yamabe invariant of the disjoint union of these two manifolds. 

There are many developments since Kobayshi's result \cite{Kobayashi1987/88}. For instance, Petean and Yun \cite{petean1999surgery} proved certain estimates of the Yamabe invariants under surgery of codimension at least $3$, where Kobayashi's inequality can be viewed as the $0$-dimensional case of their results. Schwartz \cite{schwartz2009monotonicity} showed monotonicity of two special cases of the generalized Yamabe invariants ($\sigma_{1,0}$ and $\sigma_{0,1}$) under connected sum over the boundary.

In this paper, we generalize Kobayashi's inequality and Schwartz's result to the $\lambda$-Yamabe invariants, for all $\lambda\in [0,1]$. With these tools, we prove the Main Theorem \ref{maintheorem}.

There is a substantial body of work on Yamabe invariants and we will briefly review some significant results. 

Lebrun and his collaborators \cite {akutagawa2007perelman, gursky1998yamabe,lebrun1996yamabe, lebrun2004einstein, lebrun1997kodaira} computed the Yamabe invariants for large classes of $4$-manifolds. With Lebrun's result \cite{lebrun1996yamabe}, we strengthen \ref{maintheorem} in dimension $4$, see \ref{lebrunexample}. 

An important result by Petean \cite{petean2000yamabe} is that the Yamabe invariant of any simply connected closed manifold with dimension greater than or equal to $5$ is non-negative. With Petean's result \cite{petean2000yamabe} and generalized Kobayashi's connected-sum inequality \ref{generalizedkobayshiinequality}, we show that the $\lambda$-Yamabe invariant of any simply connected clsoed manifold with dimension greater than or equal to $5$ minus finitely many disjoint balls is non-negative.

One of the mostly celebrated results is due to Bray and Neves \cite{bray2004classification}, they computed the Yamabe invariant of $\mathbb{RP}^3$, which is the only verified non-trivial case of Schoen's conjecture. Later, Akugatawa and Neves \cite{akutagawa20073} completed the classification of all closed $3$-manifolds with Yamabe invariant greater than that of $\mathbb{RP}^3$.

\begin{conjecture}[Schoen]
The Yamabe invariant of lens space is 
\[
\sigma(L(p,q))=\sigma_p:=\frac{\sigma(S^3)}{p^{2/3}},
\]
where $p,q$ are relatively prime.
\end{conjecture}

A direct corollary of \ref{maintheorem} forms an interesting contrast with Bray and Neves' results \cite{bray2004classification}.
\begin{cor}
Suppose $n\geq 3$, for any $\lambda\in [0,1]$
\begin{equation}
\sigma_{\lambda}(\mathbb RP^n\setminus(\sqcup_{i=1}^kD_i))=\sigma_{\lambda}(S^n_+),
\end{equation}
where $\{D_i\}_{i=1}^k$ is a finite collection of disjoint $n$-balls. 
\end{cor}

\begin{remark}
    As opposed to Bray-Neves, a compact manifold with boundary whose relative fundamental group is $\mathbb Z_2$ has its Yamabe invariants the same as a solid ball. 

    It is still an open problem whether the Yamabe invariant of $\mathbb{RP}^n$ is $\sigma_2$. Our result may provide a new perspective on this problem.
\end{remark}

Another famous result of the Yamabe invariant is the computation of $\sigma(S^n\times S^1)$
\[
\sigma(S^{n-1}\times S^1)=\sigma(S^n).
\]

This is proven by Schoen \cite{schoen1989variational} and Kobayashi \cite{Kobayashi1987/88} independently. They both constructed a sequence of conformal classes of metrics on $S^n\times S^1$ whose Yamabe constant converges to the Yamabe invariant of $S^{n+1}$. 
More recently, Akutagawa, Florit and Petean \cite{akutagawa2007yamabe} gave a new proof of this result by studying the Yamabe constants of Riemannian products.

Using Proposition \ref{propofkobayashi}, a corollary of the Main Theorem, we give a new proof of this famous result without using any analytical tools.

The proof of \eqref{Mainequ} could also be applied to other examples. 

In dimension $4$, we compute the $\lambda$-Yamabe invariant of 
$\#m_1\mathbb{RP}^4\#m_2\mathbb{RP}^{4}\times S^1\# m_3 \mathbb{CP}^2\#lH^4\#k S_+^4$ provided $k+l\geq 1$.

We also show that for any simply connected closed manifold $M^n$, $n\geq 5$, $\# mM^n\#m_1\mathbb{RP}^{n}\#m_2(\mathbb{RP}^{n-1}\times S^1)\# lH^n\#k S^n_+$ has non-negative $\lambda$-Yamabe invariants.  

\subsection*{Sketch of proof} 
We first show that the lower bound for $\sigma_{\lambda}(\mathbb{RP}^n)$ and \\
$\sigma_{\lambda}(\mathbb {RP}^{n-1}\times S^1)$ is $\sigma_{\lambda}(S_+^n)$,
 for any $\lambda\in (0,1]$.

To show the lower bound of $\sigma_{\lambda}(\mathbb {RP}^{n-1}\times S^1)$, we carefully analyze Schoen's construction of conformal classes of metrics whose Yamabe constant converges to $\sigma(S^n)$, and note that the metrics in that sequence can be viewed as a $2$-fold Riemannian covering of $\mathbb {RP}^{n-1}\times S^1$, which gives the lower bound of $\sigma_{\lambda}(\mathbb {RP}^{n-1}\times S^1)$.

Next, we generalize the Kobayashi's connected-sum inequality \cite{Kobayashi1987/88} and\\
Schwartz's result \cite{schwartz2009monotonicity} to $\lambda$-Yamabe invariants of compact manifolds (with or without boundary).

Combining these with Sun's continuity results \cite{sun2017yamabe}, we complete the proof of Theorem \ref{maintheorem}.

\subsection*{Organization of the paper} In Section $2$, we set up some notations and give the definition of the $\lambda$-Yamabe invariants.

Section $3$ is devoted to giving lower bounds of $\lambda$-invariants for $\mathbb{RP}^n$ and $\mathbb{RP}^{n-1}\times S^1$.

In Section $4$, we generalize the Kobayashi's connected-sum inequality and Schwartz's result to $\lambda$-Yamabe invariants of compact manifolds (with or without boundary), and give a new proof of $\sigma(S^{n-1}\times S^1)=\sigma(S^n)$ for any $n\geq 2$.

Section $5$ is devoted to the proof of the continuity of the $\lambda$-Yamabe constants in $\lambda$. We adopt a different but equivalent constrain condition from Sun's result \cite{sun2017yamabe}, and rewrite his proof.

In Section $6$, we complete the proof of the Main Theorem \ref{maintheorem}. There are two ingredients in the proof, one is the lower bound of $\lambda$-Yamabe invariants for $\mathbb{RP}^n$ and $\mathbb{RP}^{n-1}\times S^1$, which are proven in Section $3.1$ and Section $3.2$; the other is the generalized Kobayashi's connected-sum inequality, which is stated and proven in Section 5.

Finally, Section $7$ provides other examples and applications of the generalized Kobayashi's connected-sum inequality. 
\vspace{0.5em}

\leftline{\textbf{Acknowledgement:}} I would like to thank my advisor Xin Zhou for
suggesting this problem and for many helpful discussions.
\section{Preliminaries}
Let $(M^n, g)$ be a $n$-dimensional compact Riemannian manifold with (possibly empty) boundary $\partial M$. We adopte the definitions in\cite{Escobar96} as follows.

For any $u\in H^1(M,g)$, define the energy of $u$, $E_M(u)$, by
\[
E_M(u)=\int_{M}\frac{4(n-1)}{n-2}|\nabla u|_g^2dV_g+\int_{M}R_gu^2dV_g+2(n-1)\int_{\partial M}H_gu^2d\sigma_g,
\]
where $H_g(x)$ is the mean curvature of the boundary at $x\in\partial M$, $dV_g$ and $d\sigma_g$ represent the Riemannian measure on $M$ and $\partial M$ induced by the metric $g$.

Given any $(a,b)\in \mathbb R_{\geq 0}\times\mathbb R\setminus\{(0,0)\}$, define the constraint set
\[
C_{a,b}(M,g):=\{u\in H^1(M,g): a\int_{M}|u|^pdV_g+b\int_{\partial M}|u|^qd\sigma_g=1, u>0\},
\]
where $p=2n/(n-2)$, $q=2(n-1)/(n-2)$.
\begin{defn}
The generalized Yamabe constant for a compact Riemannian manifold $(M^n,g)$ is defined as
\[
Y_{a,b}(M,\partial M, [g])=\inf_{u\in C_{a,b}(M,g)}E_M(u).
\]
\end{defn}
\begin{remark}
Both the energy functional and the constraint set depend on the metric $g$.
\end{remark}

Similarly, we can generalize the definition of Yamabe invariant as
\begin{defn}
The generalized Yamabe invariant is defined as
\[
\sigma_{a,b}(M):=\sup_{[g]\in \mathcal C}Y_{a,b}(M,\partial M,[g]),
\]
where $\mathcal C$ is the set of all conformal classes on $M$.
\end{defn}

In this paper, we focus on the cases when $a, b$ are both non-negative, and define the $\lambda$-Yamabe constant and $\lambda$-Yamabe invariant as
\begin{defn}\label{lambdayamabe}
\[
Y_{\lambda}(M,\partial M, [g]):=Y_{\lambda,1-\lambda}(M,\partial M,[g])\quad \sigma_{\lambda}(M):=\sigma_{\lambda,1-\lambda}(M),\quad \lambda\in [0,1].
\]
\end{defn}
We call $Y_{\lambda}$ the $\lambda$-Yamabe constant and $\sigma_{\lambda}$ the $\lambda$-Yamabe invariant.

\section{Lower Bound Estimate}
In this section, we give the lower bounds of the $\lambda$-Yamabe invariants for $\mathbb{RP}^n$ and $\mathbb{RP}^{n-1}\times S^1$ respectively.
\subsection{Lower bound of \texorpdfstring{$\mathbb{RP}^3$}{}}

We first introduce a lemma to bound the Yamabe invariant of a closed manifold from below by the Yamabe invariant of a compact manifold with boundary. As an application, we can derive the lower bound of $\sigma_{\lambda}(\mathbb{RP}^n)$.
\begin{lemma}
Assume that $\lambda\neq 0$. Suppose that $(M^n,g)$ is a closed Riemannian manifold with non-negative $\lambda$-Yamabe constant, and $\Sigma^{n-1}\subset M^n$ is a closed embedded minimal surface, then $N=M\setminus \Sigma$ is a compact manifold with boundary, and the follwoing inequality holds
\begin{equation}
Y_{\lambda}(M,[g])\geq Y_{\lambda}(N,\partial N,[g]).
\end{equation}
\end{lemma}
\begin{proof}
For any $u\in C_{\lambda,1-\lambda}(M,g)$, we have
\begin{displaymath}
\lambda \int_{N}|u|^pdV_g+(1-\lambda)\int_{\partial N}|u|^qd\sigma_g=\lambda\int_{M}|u|^pdV_g+(1-\lambda)\int_{\partial N}|u|^qd\sigma_g\geq 1.
\end{displaymath}
Suppose $cu\in C_{\lambda,1-\lambda}(N^n,g)$ where $c>0$ is a constant, then $c\leq 1$. With the fact that $Y_{\lambda}(M,[g])\geq 0$, and $\partial N$ is minimal,  we obtain
\begin{displaymath}
\begin{split}
E_{M}(u)&=\int_{M}\frac{4(n-1)}{n-2}|\nabla u|_g^2dV_g+\int_{M}R_gu^2dV_g\\
&=\int_{N}\frac{4(n-1)}{n-2}|\nabla u|_g^2dV_g+\int_{N}R_gu^2dV_g\\
&=E_{N}(u)\\
&\geq E_{N}(cu)\\
&\geq Y_{\lambda}(N,\partial N,[g]),\quad \forall u\in C_{\lambda,1-\lambda}(M,g),
\end{split}
\end{displaymath}
thus
\begin{equation}
Y_{\lambda}(M,[g])=\inf_{u\in C_{\lambda,1-\lambda}(M)}E_M(u)\geq Y_{\lambda}(N,\partial N,[g]).
\end{equation}
\end{proof}
Let $(M^n,g)$ be $(\mathbb{RP}^n,g_0)$, and $g_0$ be the standard round metric induced on $\mathbb{RP}^n$, then there exists a totally geodesic $\mathbb{RP}^{n-1}$, such that $(\mathbb{RP}^n\setminus \Sigma,g_0)=(S^n_+,g_0)$. We obtain the lower bound
\begin{equation}
\sigma_{\lambda}(\mathbb{RP}^n)\geq Y_{\lambda}(\mathbb{RP}^n,[g_0])\geq Y_{\lambda}(S^n_+,[g_0])=\sigma_\lambda(S^n_+).
\end{equation}

\subsection{Lower bound of \texorpdfstring{$\sigma_{\lambda}(\mathbb{RP}^{n-1}\times S^1)$}{}}
Given a closed manifold $(M,g)$ and its $k$-fold Riemannian covering $(M_k,g_k)$, we can bound $Y_{\lambda}(M,[g])$ from below by $Y_{\lambda}(M_k,g_k)$. As an application, we give the lower bound of $\sigma_{\lambda}(\mathbb{RP}^{n-1}\times S^1)$.
\begin{lemma}\label{kfoldcovering}
Suppose $(M_k,g_k)$ is a $k$-fold Riemannian covering of a closed Riemannian manifold $(M,g)$, then we have
\begin{equation}
Y_{\lambda}(M,[g])\geq \frac{Y_{\lambda}(M_k,[g_k])}{k^{2/n}}.
\end{equation}
\end{lemma}
\begin{proof}
For any $u\in C_{\lambda,1-\lambda}(M)$, we could lift $u$ to $(M_k,g_k)$, and denote it as $u_k$. 

Then
\begin{displaymath}
\lambda \int_{M_k}|u_k|^pdV_{g_k}=k,
\end{displaymath}
and 
\begin{displaymath}
E_{M_k}(u_k)\geq k^{\frac{2}{p}}Y_{\lambda}(M_k,[g_k]).
\end{displaymath}
Since $(M_k,g_k)$ is a $k$-fold Riemannian covering of $(M,g)$, we have
\begin{displaymath}
E_M(u)=\frac{1}{k}E_{M_k}(u_k)\geq \frac{Y_{\lambda}(M_k,[g_k])}{k^{2/n}}.
\end{displaymath}
Thus
\begin{displaymath}
Y_{\lambda}(M,[g])\geq \frac{Y_{\lambda}(M_k,[g_k])}{k^{2/n}}.
\end{displaymath}
\end{proof}
\begin{remark}
As a direct corollary, we obtain the lower bound of the Yamabe invariant of Lens space
\begin{displaymath}
\sigma_1(L(p,q))\geq \frac{\sigma_1(S^3)}{p^{2/3}}.
\end{displaymath}
\end{remark}

To show the lower bound of $\sigma_{\lambda}(\mathbb{RP}^{n-1}\times S^1)$, we need to review the proof of $\sigma_{\lambda}(S^{n-1}\times S^1)=\sigma_{\lambda}(S^n)$, which was given independently by Schoen \cite{schoen1989variational} and Kobayashi \cite{Kobayashi1987/88}. More recently, Akustagawa, Florit and Peaten \cite{akutagawa2007yamabe} gave a new proof.  

Here we use Schoen's construction, where he constructed
\begin{equation}
\lim_{l\to\infty}Y_{\lambda}(S^{n-1}\times S^1,[g_l])=\sigma_{\lambda}(S^n),
\end{equation}
and $g_l|_{S^{n-1}\times\{t\}}$ is the standard round metric for any $t\in S^1$. Then $(S^{n-1}\times S^1,g_l)$ can be viewed as a $2$-fold Riemannian covering of $(\mathbb{RP}^{n-1}\times S^1,g_l)$, with the help of Lemma \ref{kfoldcovering}, we can deduce that
\begin{equation}
\sigma_{\lambda}(\mathbb{RP}^{n-1}\times S^1)\geq \lim_{l\to\infty}\frac{1}{2^{2/n}}Y_{\lambda}(S^{n-1}\times S^1,[g_l])=\frac{1}{2^{2/n}}\sigma_{\lambda}(S^n).
\end{equation}
It is known by \cite{Escobar96} that 
\begin{equation}
\sigma_{\lambda}(S^n)=Y_{\lambda}(S^n,[g_0]),\quad \sigma_{\lambda}(S^n_+)=Y_{\lambda}(S^n_+,[g_0]).
\end{equation}
For any $u\in C_{\lambda, 1-\lambda}(S^n_+)$, we extend $u$ to $(S^n,g_0)$ by reflection, and denote it by $\tilde{u}$, then we have $\tilde{u}\in \text{Lip}(S^n)$, $E_{S^n_+}(u)=\frac{1}{2}E_{S^n}(\tilde{u})$, and by straightforward computation, we get
\begin{equation}
\frac{1}{2^{2/n}}\sigma_{\lambda}(S^n)\geq \sigma_{\lambda}(S^n_+).
\end{equation}
Combining the above results, we obtain
\begin{equation}
\sigma_{\lambda}(\mathbb{RP}^{n-1}\times S^1)\geq \sigma_{\lambda}(S^n_+).
\end{equation}

\section{Generalized Kobayashi's inequality}
Kobayashi \cite{Kobayashi1987/88} established a monotonicity formula of Yamabe invariant over connected sums of closed manifolds. Here, we extend it to compact manifolds (with or without boundary) and establish the results for $\lambda$-Yamabe invariant.
\begin{lemma}[Generalized Kobayashi's Lemma]\label{generalizedkobayshilemma} Suppose $(M_1,g_1)$ and $(M_2,g_2)$
are two $n$-dimensional compact Riemanniamanifolds, and we further assume that
 $Y_{\lambda}(M_i,[g_i]) \geq 0$ for some $\lambda\in (0,1]$, $i=1,2$, then we obtain
\begin{displaymath}
Y_{\lambda}(M_1\sqcup M_2, [g_1\sqcup g_2])\geq \min\{Y_{\lambda}(M_1,[g_1]), Y_{\lambda}(M_2,[g_2])\}.
\end{displaymath}
\end{lemma}

\begin{remark}
If we require both $M_1$ and $M_2$ to be manifolds with non-empty boundaries, \ref{generalizedkobayshilemma} holds for $\lambda\in [0,1]$.
\end{remark}

\begin{proof}
Denote $u=[u_1,u_2]\in C^{\infty}(M_1\sqcup M_2)$, where $u_1\in C^{\infty}(M_1)$, $u_2\in C^{\infty}(M_2)$.

Also, suppose $u\in C_{\lambda,1-\lambda}(M_1\sqcup M_2)$, i.e.
\begin{displaymath}
\begin{split}
\lambda\int_{M_1}|u_1|^pdV_{g_1}&+(1-\lambda)\int_{\partial M_1}|u_1|^qd\sigma_{g_1}\\
&+\lambda\int_{M_1}|u_2|^pdV_{g_2}+(1-\lambda)\int_{\partial M_2}|u_2|^qd\sigma_{g_2}=1.
\end{split}
\end{displaymath}

Let 
\begin{displaymath}
\alpha=\lambda\int_{M_1}|u_1|^pdV_{g_1}+(1-\lambda)\int_{\partial M_1}|u_1|^qd\sigma_{g_1},
\end{displaymath}
then
\begin{displaymath}
1-\alpha=\lambda\int_{M_2}|u_2|^pdV_{g_2}+(1-\lambda)\int_{\partial M_2}|u_2|^qd\sigma_{g_2}.
\end{displaymath}
Fix $\alpha$, and suppose $C_1(\alpha)$ is the constant such that $C_1(\alpha)u_1\in C_{\lambda,1-\lambda}(M_1)$, then we have
\begin{displaymath}
E_{M_1}(C_1(\alpha)u_1)\geq Y_{\lambda}(M_1,[g_1]),
\end{displaymath}
and
\begin{displaymath}
\alpha C_1(\alpha)^q\leq 1\leq \alpha C_1(\alpha)^p.
\end{displaymath}
The above estimates imply that
\begin{displaymath}
E_{M_1}(u_1)=\frac{E_{M_1}(C_1(\alpha))}{C_1(\alpha)^2}\geq \alpha^{\frac{n-2}{n-2}}Y_{\lambda}(M_1,[g_1]).
\end{displaymath}
Similarly, we could deduce 
\begin{displaymath}
E_{M_2}(u_2)\geq (1-\alpha)^{\frac{n-2}{n-1}}Y_{\lambda}(M_2,[g_2]).
\end{displaymath}
Combining the above results, we know 
\begin{displaymath}
E_{M_1\sqcup M_2}(u)=E_{M_1}(u_1)+E_{M_2}(u_2)\geq f(\alpha),
\end{displaymath}
where $f(\alpha)=Y_{\lambda}(M_1,[g_1])\alpha^{\frac{n-1}{n-2}}+Y_{\lambda}(M_2,[g_2])\alpha^{\frac{n-1}{n-2}}$. 

By straightforward computations, we know
\begin{displaymath}
f(\alpha)\geq \min\{Y_{\lambda}(M_1,[g_1]), Y_{\lambda}(M_2,[g_2])\}, \forall \alpha\in [0,1].
\end{displaymath}
This implies
\begin{displaymath}
Y_{\lambda}(M_1\sqcup M_2,[g_1\sqcup g_2])\geq \min\{Y_{\lambda}(M_1,\partial M_1, [g_1]), Y_{\lambda}(M_2,\partial M_2,[g_2])\}.
\end{displaymath}
\end{proof}

Now we generalize the typical Kobayashi's inequality to $\lambda$-Yamabe invariant, and to manifolds with or without boundary.

\begin{thm}[Generalized Kobayshi's inequality]\label{generalizedkobayshiinequality}
Suppose $M_1$ and $M_2$ are compact manifolds (with or without boundary) of dimension $n\geq 3$, then $\forall \lambda\in (0,1]$, we have
\begin{equation}
\sigma_{\lambda}(M_1\# M_2)\geq \sigma_{\lambda}(M_1\sqcup M_2).
\end{equation}
\end{thm}

\begin{remark}
Suppose both $M_1$ and $M_2$ are manifolds with boundary, then Lemma \ref{generalizedkobayshiinequality} holds for $\lambda\in[0,1]$.
\end{remark}

\begin{proof}
Let $M=M_1\sqcup M_2$, $\epsilon$ be an arbitrary positive number, then there exists a conformal class $[g]$ of $M$, such that 
\begin{equation}
Y_{\lambda}(M,\partial M, [g])\geq \sigma_{\lambda}(M)-\epsilon.
\end{equation}
Let $p_1\in M_1$, $p_2\in M_2$, be points in the interior of $M_1$ and $M_2$ respectively. We can also assume $[g]$ is conformally flat near $p_1$ and $p_2$; see \cite{Kobayashi1987/88}. Then there exists a function $\varphi\in C^{\infty}(M\setminus\{p_1,p_2\})$, such that $\tilde{g}=e^{\varphi}g$ is a complete metric of $M\setminus\{p_1,p_2\}$, the boundary of $M$ is minimal, and each of two ends is isometric to the standard half infinite cylinder $[0,\infty)\times S^{n-1}$. 

We write
\begin{displaymath}
(M\setminus\{p_1,p_2\},\tilde{g})=[0,\infty)\times S^{n-1}\cup (\tilde{M},\tilde{g})\cup [0,\infty)\times S^{n-1},
\end{displaymath}
where $\tilde{M}$ is the complement of the two cylinders.

We can glue $(\tilde{M},\tilde{g})$ and $[0,l]\times S^{n-1}$, the product of the interval of length $l$ with the unit $(n-1)$ sphere, along the boundaries to get a smooth Riemannina manifold $(\bar{M},g_l)$,
\begin{displaymath}
(\bar{M},g_l)=(\tilde{M},\tilde{g})\cup [0,l]\times S^{n-1}.
\end{displaymath}
Take $f_l\in C_{\lambda,1-\lambda}(\bar{M})$, such that $E_{\bar{M}}(f_l)\leq Y_{\lambda}(\bar{M},\partial \bar{M},[g_l])+\frac{1}{l+1}$.

We have the following lemma
\begin{lemma}
There exists $t_l\in [0,l]$, and a constant $A$ not depending on $l$, such that
\begin{equation}
\int_{\{t_l\}\times S^{n-1}}(|\nabla f_l|^2+f_l^2)dV_{S^{n-1}}<\frac{A}{l}.
\end{equation}
\end{lemma}
\begin{proof}
By definition, 
\begin{equation}
\begin{split}
E_{\bar{M}}(f_l)=&\int_{\bar{M}}|\nabla f_l|^2dV_{g_l}+a(n)\int_{M_l}R_{g_l}f_l^2dV_{g_l}+b(n)\int_{\partial \bar{M}}H_{g_l}f_l^2d\sigma_{g_l}\\
=&\int_{\bar{M}}|\nabla f_l|^2dV_{g_l}+a(n)\int_{M_l}R_{g_l}f_l^2dV_{g_l}\\
=&\int_{\tilde{M}}|\nabla f_l|^2dV_{g_l}+a(n)\int_{\tilde{M}}R_{\tilde{g}}f_l^2dV_{g_l}\\
&+\int_{[0,l]\times S^{n-1}}|\nabla f_l|^2dV_{g_l}+a(n)n(n-1)\int_{[0,l]\times S^{n-1}}f_l^2dV_{g_l}\\
\geq& a(n)\min\{0,\min_{x\in\tilde{M}}R_{\tilde{g}}(x)\}\int_{\tilde{M}}f_l^2dV_{\tilde{g}}+\tilde{a}(n)\int_{[0,l]\times S^{n-1}}f_l^2dV_{g_l}\\
&+\int_{[0,l]\times S^{n-1}}|\nabla f_l|^2dV_{g_l}
\end{split}
\end{equation}
here we used the fact the boundary of $\bar{M}$ is minimal.

Also note that
\begin{displaymath}
\lambda\int_{\tilde{M}}|f_l|^pdV_{\tilde{g}}\leq 1,
\end{displaymath}
by Hölder's inequality, we know
\begin{displaymath}
\int_{\tilde{M}}f_l^2dV_{\tilde{g}}\leq \frac{1}{\lambda^{1-\frac{2}{n}}}\text{Vol}(\tilde{M})^{2/n}.
\end{displaymath}
This implies
\begin{displaymath}
\int_{[0,l]\times S^{n-1}}|\nabla f_l|^2dV_{g_l}+\int_{[0,l]\times S^{n-1}}f_l^2dV_{g_l}\leq Y_{\lambda}(\bar{M},[g_l])+\frac{1}{l+1}+A_1,
\end{displaymath}
by Mean Value Theorem, we complete the proof.
\end{proof}
Cut off $\bar{M}$ on the section $\{t_l\}\times S^{n-1}$, and attach 2 infinite cylinders to it, so that $(M\setminus \{p_1,p_2\},\tilde{g})$ reappears:
\begin{displaymath}
(M\setminus\{p_1,p_2\},\tilde{g})=[0,\infty)\times S^{n-1}\cup (\bar{M}\setminus \{t_l\}\times S^{n-1},g_l)\cup [0,\infty)\times S^{n-1}.
\end{displaymath}
Define a Lipshitz function on $\bar{M}\setminus\{p_1,p_2\}$ as
\begin{displaymath}
F_l(x)=\left\{\begin{array}{ccc}
f_l(x) & x\in \bar{M}\setminus \{t_l\}\times S^{n-1}\\
(1-t)\tilde{f_l}(y) & x=(t,y)\in [0,1]\times S^{n-1}\\
0 & x=(t,y)\in [1,\infty)\times S^{n-1}
\end{array}\right.
\end{displaymath}
where $\tilde{f_l}=f_l|_{\{t_l\}\times S^{n-1}}$.
Then we compute that
\begin{displaymath}
E_{M\setminus\{p_1,p_2\}}(F_l)\leq Y_{\lambda}(\bar{M},\partial\bar{M}, [g_l])+\frac{B}{l},
\end{displaymath}
where $B$ is constant independent of $l$.

Also, we have
\begin{displaymath}
\lambda\int_{M\setminus\{p_1,p_2\}}|F_l|^pdV_{\tilde{g}}+(1-\lambda)\int_{\partial M}|F_l|^qd\sigma_{\tilde{g}}\geq \lambda \int_{\bar{M}}|f_l|^pdV_{g_l}+(1-\lambda)\int_{\partial \bar{M}}|f_l|^q d\sigma_{g_l}=1
\end{displaymath}
which indicates that
\begin{displaymath}
Y_\lambda(M,\partial{M}, [g])\leq Y_{\lambda}(\bar{M},\partial\bar{M}, [g_l])+\frac{B}{l}
\end{displaymath}
After passing $l$ to $\infty$, we can obtain that
\begin{equation}
\sigma_{\lambda}(M)-\epsilon\leq Y_{\lambda}(M,\partial M, [g])\leq \sigma_{\lambda}(M_1\# M_2).
\end{equation}
Since $\epsilon$ can be arbitrarily small, we obtain that
\begin{displaymath}
\sigma_{\lambda}(M_1\sqcup M_2)\leq \sigma_{\lambda}(M_1\# M_2).
\end{displaymath}
As a direct corollary, suppose $\sigma_{\lambda}(M_1), \sigma_{\lambda}(M_2)\geq 0$, we have
\begin{equation}
\sigma_{\lambda}(M_1\# M_2)\geq \min\{\sigma_{\lambda}(M_1),\sigma_\lambda(M_2)\}.
\end{equation}
\end{proof}

Similarly, we can obtain
\begin{prop}\label{propofkobayashi}
Suppose $M^n$ is a compact manifold (with or wthout boundary) of dimension $n\geq 3$, and $\sigma_{\lambda}(M)\geq 0$, then $\forall \lambda\in (0,1]$, we have
\begin{equation}\label{selfconnectsum}
\sigma_{\lambda}(M\# (S^{n-1}\times S^1))\geq \sigma_{\lambda}(M).
\end{equation}
\end{prop} 
\begin{remark} $ $
\begin{enumerate}
\item Suppose $M$ is a compact $n$-manifold with boundary, then \eqref{selfconnectsum} holds for $\lambda\in [0,1]$.
\item $M\# S^{n-1}\times S^1$ can be regarded as the ''connect sum'' of $M$ with itself, applying the same proof in Generalized Kobayashi's inequality, \ref{propofkobayashi} could be easily proven.

\end{enumerate}
\end{remark}

With \ref{propofkobayashi}, we give a new proof of the famous result that $S^{n-1}\times S^1$ has the same Yamabe invariant as $S^n$ without using analytical tools.
\begin{cor}
For any $n\geq 3$, $\lambda\in (0,1]$,
\begin{equation}
\sigma_{\lambda}(S^{n-1}\times S^1)= \sigma_{\lambda}(S^n).
\end{equation}
\end{cor}
\begin{proof}
Take $M^n=S^n$, then $M\# S^{n-1}\times S^1=S^{n-1}\times S^1$, applying \ref{propofkobayashi} we could obtain
\[
\sigma_{\lambda}(S^{n-1}\times S^1)\geq \sigma_{\lambda}(S^n).
\]
Since it is well-known $\sigma_{\lambda}(M)\leq \sigma_{\lambda}(S^n)$ for any compact closed manifold $M$, we conclude that
\[
\sigma_{\lambda}(S^{n-1}\times S^1)= \sigma_{\lambda}(S^n).
\]
\end{proof}

Schwartz \cite{schwartz2009monotonicity} gave a generalization of the classic Kobayashi's inequality to monotonicity of Yamabe invariant over connect sums along boundaries, as a direct corollary, he showed 
\begin{displaymath}
\sigma_{\lambda}(H^n)=\sigma_{\lambda}(S^n_+),\quad \lambda\in \{0,1\}.
\end{displaymath} 
Where $H^n$ is any $n$-dimensional handle body.

Here, we extend his results to $\lambda\in (0,1)$, we use $\tilde{\#}$ to denote connect sums along boundaries, and as a direct corollary, we show that 
\begin{equation}
\sigma_{\lambda}(H^n)=\sigma_{\lambda}(S^n_+)\quad \forall \lambda\in [0,1].
\end{equation}
\begin{thm}\label{connectsumalongboundary}
Suppose $M_1$ and $M_2$ are compact manifolds with boundaries of dimension $n\geq 3$, then $\forall \lambda\in [0,1]$, we have
\begin{equation}
\sigma_{\lambda}(M_1\tilde{\#} M_2)\geq \sigma_{\lambda}(M_1\sqcup M_2).
\end{equation}
\end{thm}
\begin{proof}
Let $M=M_1\sqcup M_2$, $\epsilon$ be an arbitrary positive number, then there exists a conformal class $[g]$ of $M$, such that 
\begin{equation}
Y_{\lambda}(M,\partial M, [g])\geq \sigma_{\lambda}(M)-\epsilon.
\end{equation}
Take $p_1\in \partial M_1$, $p_2\in \partial M_2$, there exists a function $\varphi\in C^{\infty}(M\setminus\{p_1,p_2\})$, such that $\tilde{g}=e^{\varphi}g$ is isometric to the standard half infinite hemi-cylinder $[0,\infty)\times S_+^{n-1}$ near the removable points. We write
\begin{displaymath}
(M\setminus\{p_1,p_2\},\tilde{g})=[0,\infty)\times S_+^{n-1}\cup (\tilde{M},\tilde{g})\cup [0,\infty)\times S_+^{n-1},
\end{displaymath}
where $\tilde{M}$ is the complement of the two hemi-cylinders.

We can glue $(\tilde{M},\tilde{g})$ and $[0,l]\times S_+^{n-1}$, the product of the interval of length $l$ with the unit $(n-1)$ sphere, along the boundaries to get a smooth Riemannina manifold $(\bar{M},g_l)$,
\begin{displaymath}
(\bar{M},g_l)=(\tilde{M},\tilde{g})\cup [0,l]\times S_+^{n-1}.
\end{displaymath}
Take $f_l\in C_{\lambda,1-\lambda}(\bar{M})$, such that $E_{\bar{M}}(f_l)\leq Y_{\lambda}(\bar{M},\partial \bar{M}, [g_l])+\frac{1}{l+1}$.

We have the following lemma
\begin{lemma}
There exists $t_l\in [0,l]$, and a constant $A$ not depending on $l$, such that
\begin{equation}
\int_{\{t_l\}\times S_+^{n-1}}(|\nabla f_l|^2+f_l^2)dV_{S^{n-1}}<\frac{A}{l}.
\end{equation}
\end{lemma}
\begin{proof}
By definition, 
\begin{equation}
\begin{split}
E_{\bar{M}}(f_l)=&\int_{\bar{M}}|\nabla f_l|^2dV_{g_l}+a(n)\int_{M_l}R_{g_l}f_l^2dV_{g_l}+b(n)\int_{\partial \bar{M}}H_{g_l}f_l^2d\sigma_{g_l}\\
=&\int_{\tilde{M}}|\nabla f_l|^2dV_{g_l}+a(n)\int_{\tilde{M}}R_{\tilde{g}}f_l^2dV_{g_l}+b(n)\int_{\partial \tilde{M}\setminus \{\{0\}\times S^{n-1}_+, \{l\}\times S^{n-1}_+\}}H_{\tilde{g}}f_l^2d\sigma_{\tilde{g}}\\
&+\int_{[0,l]\times S^{n-1}}|\nabla f_l|^2dV_{g_l}+a(n)n(n-1)\int_{[0,l]\times S_+^{n-1}}f_l^2dV_{g_l}\\
\geq& a(n)\min\{0,\min_{x\in\tilde{M}}R_{\tilde{g}}(x)\}\int_{\tilde{M}}f_l^2dV_{\tilde{g}}\\
&+b(n)\min\{0,\min_{x\in\partial \tilde{M}\setminus\{\{0\}\times S^{n-1}_+, \{l\}\times S^{n-1}_+\}}H_{\tilde{g}}(x)\}\int_{\partial\tilde{M}}f_l^2d\sigma_{\tilde{g}}+\tilde{a}(n)\int_{[0,l]\times S_+^{n-1}}f_l^2dV_{g_l}\\
&+\int_{[0,l]\times S_+^{n-1}}|\nabla f_l|^2dV_{g_l}
\end{split}
\end{equation}
here we used the fact the boundary of $\bar{M}$ is minimal.

Notice that
\begin{displaymath}
\lambda\int_{\tilde{M}}|f_l|^pdV_{\tilde{g}}\leq 1,
\end{displaymath}
by Hölder's inequality, we know
\begin{displaymath}
\int_{\tilde{M}}f_l^2dV_{\tilde{g}}\leq \frac{1}{\lambda^{1-\frac{2}{n}}}\text{Vol}(\tilde{M})^{2/n}.
\end{displaymath}
Similarly, we have
\begin{displaymath}
\int_{\partial\tilde{M}\setminus \{\{0\}\times S^{n-1}_+,\{l\}\times S^{n-1}_+\}}f_l^2\leq (1-\lambda)^{\frac{n-2}{n-1}}\text{Area}(\partial\tilde{M}\setminus\{\{0\}\times S^{n-1}_+, \{l\}\times S^{n-1}_+\})^{\frac{1}{n-1}}.
\end{displaymath}
This implies
\begin{displaymath}
\int_{[0,l]\times S^{n-1}}|\nabla f_l|^2dV_{g_l}+\int_{[0,l]\times S^{n-1}}f_l^2dV_{g_l}\leq Y_{\lambda}(\bar{M},[g_l])+\frac{1}{l+1}+A_1,
\end{displaymath}
by the Mean Value Theorem, we complete the proof.
\end{proof}
Cut off $\bar{M}$ on the section $\{t_l\}\times S_+^{n-1}$, and attach 2 infinite cylinders to it, so that $(M\setminus \{p_1,p_2\},\tilde{g})$ reappears:
\begin{displaymath}
(M\setminus\{p_1,p_2\},\tilde{g})=[0,\infty)\times S_+^{n-1}\cup (\bar{M}\setminus \{t_l\}\times S_+^{n-1},g_l)\cup [0,\infty)\times S_+^{n-1}.
\end{displaymath}
Define a Lipshitz function on $\bar{M}\setminus\{p_1,p_2\}$ as
\begin{displaymath}
F_l(x)=\left\{\begin{array}{ccc}
f_l(x) & x\in \bar{M}\setminus \{t_l\}\times S_+^{n-1}\\
(1-t)\tilde{f_l}(y) & x=(t,y)\in [0,1]\times S_+^{n-1}\\
0 & x=(t,y)\in [1,\infty)\times S_+^{n-1}
\end{array}\right.
\end{displaymath}
where $\tilde{f_l}=f_l|_{\{t_l\}\times S_+^{n-1}}$.
Then we could compute that
\begin{displaymath}
E_{M\setminus\{p_1,p_2\}}(F_l)\leq Y_{\lambda}(\bar{M},\partial \bar{M}, [g_l])+\frac{B}{l},
\end{displaymath}
where $B$ is constant independent of $l$.

Also, we have
\begin{displaymath}
\lambda\int_{M\setminus\{p_1,p_2\}}|F_l|^pdV_{\tilde{g}}+(1-\lambda)\int_{\partial M}|F_l|^qd\sigma_{\tilde{g}}\geq \lambda \int_{\bar{M}}|f_l|^pdV_{g_l}+(1-\lambda)\int_{\partial \bar{M}}|f_l|^q d\sigma_{g_l}=1.
\end{displaymath}
which indicates that
\begin{displaymath}
Y_\lambda(M,\partial M, [g])\leq Y_{\lambda}(\bar{M},\partial\bar{M}, [g_l])+\frac{B}{l}.
\end{displaymath}
After passing $l$ to $\infty$, we can obtain that
\begin{equation}
\sigma_{\lambda}(M)-\epsilon\leq Y_{\lambda}(M,\partial M, [g])\leq \sigma_{\lambda}(M_1\tilde{\#} M_2).
\end{equation}
Since $\epsilon$ could be arbitrarily small, we could obtain that
\begin{displaymath}
\sigma_{\lambda}(M_1\sqcup M_2)\leq \sigma_{\lambda}(M_1\tilde{\#} M_2).
\end{displaymath}
As a direct corollary, supposing $\sigma_{\lambda}(M_1), \sigma_{\lambda}(M_2)\geq 0$, we have
\begin{equation}
\sigma_{\lambda}(M_1\tilde{\#} M_2)\geq \min\{\sigma_{\lambda}(M_1),\sigma_\lambda(M_2)\}.
\end{equation}
\end{proof}

Escobar \cite{Escobar96} had proved the following Theorem.
\begin{thm}[Escobar]\label{Escobarupperbound}
Suppose $M$ is any $n$-dimensional compact manifold with boundary, and $\forall \lambda\in [0,1]$, then we have
\begin{equation}
\sigma_{\lambda}(M)\leq \sigma_{\lambda}(S^n_+).
\end{equation}
\end{thm}
With this fact and \ref{connectsumalongboundary}, we deduce that
\begin{cor}
\begin{displaymath}
\sigma_{\lambda}(H^n)=\sigma_{\lambda}(S^{n-1}_+),\quad \forall\lambda\in[0,1].
\end{displaymath}
\end{cor}

\section{Continuity of $Y_{\lambda,1-\lambda}(M,\partial M,[g])$ in $\lambda$}
Given $(M,g)$, any compact manifold with boundary, Escobar \cite{Escobar96} proved the continuity $Y_{a,b}(M,\partial M,[g])$ in $b$, provided that $a$ is fixed and positive. In the spirit of Escobar's proof, Sun \cite{sun2017yamabe} proved the continuity of $Y_{a,b}(M,\partial M,[g])$ in $K$, where $K=\{a,b\in\mathbb{R}: a\geq 0, b\geq 0\}\setminus\{(0,0)\}$, under the assumption $Y_{1,0}(M,\partial M,[g])\geq 0$. 

Since we adopt a different but equivalent constrain condition from Sun's proof \cite{sun2017yamabe}, we rewrite the continuity of $Y_{\lambda,1-\lambda}(M,\partial M,[g])$ as follows.
\begin{prop} 
Suppose $(M,g)$ is a compact manifold with boundary. Assume that $Y_{1,0}(M,\partial M,[g])\geq 0$ and $(a,b)\in K$, then $Y_{a,b}(M,\partial M,[g])$ is non-increasing in $a$ for any fixed $b$, as well as in $b$ for any fixed $a$, and is continuous in $K$.
\end{prop}
\begin{proof}
Suppose $b>0$ fixed, and $a\geq 0$. 

For any $0\leq a_1\leq a_2$, and any $\epsilon>0$, there exists $u_1\in C_{a_1,b}(M)$ such that
\[
Y_{a_1,b}(M,\partial M,[g])\leq E(u)\leq Y_{a_1,b}(M,\partial M,[g])+\epsilon.
\]
Since $0\leq a_1\leq a_2$, we have
\[
a_2\int_{M}|u_1|^pdV_g+b\int_{\partial M}b|u|^qd\sigma_g\geq 1,
\]
which implies
\[
Y_{a_2,b}(M,\partial M,[g])\leq E(u).
\]
Therefore, we conclude that
\[
Y_{a_2,b}(M,\partial M,[g])\leq Y_{a_1,b}(M,\partial M,[g]),
\]
provided $0\leq a_1\leq a_2$.

Now we show the continuity of $Y_{a,b}(M,\partial M,[g])$ in $a$.

Since $Y_{1,0}(M,\partial M,[g])\geq 0$, we may assume that $R_g=0$ and $H_g\geq 0$. Let $\{a_n\}_{n=1}^{\infty}$ be any non-negative sequence that approaches $a$.

For any $\epsilon>0$, there exists $u\in C_{a,b}(M)$, such that
\[
E_M(u)\leq Y_{a,b}(M,\partial M,[g])+\epsilon.
\]
For this particular $u$, we have $\{\lambda_n\geq 0\}_{n=1}^{\infty}$, such that $\lambda_nu\in C_{a_n,b}(M)$, and $\lim_{n\to\infty}\lambda_n=1$. Then we have
\[
Y_{a_n,b}(M,\partial M,[g])\leq E(\lambda_n u)=\lambda_n^2 E_M(u)\leq \lambda_n^2(Y_{a,b}(M,\partial M,[g])+\epsilon).
\]
Take $n\to\infty$, we can obtain
\[
\lim_{n\to\infty}Y_{a_n,b}(M,\partial M,[g])\leq Y_{a,b}(M,\partial M,[g])+\epsilon.
\]
Since $\epsilon$ can be arbitrarily small, we have
\[
\lim_{n\to\infty}Y_{a_n,b}(M,\partial M,[g])\leq Y_{a,b}(M,\partial M,[g]).
\]

For fixed $a_n$, given any $\epsilon_1>0$, there exists $u_n\in C_{a_n,b}(M)$ such that
\[
E(u_n)\leq Y_{a_n,b}(M,\partial M,[g])+\epsilon.
\]
Let $a_0=\inf_{n}a_n$, by monotonicity of $Y_{a,b}(M,\partial M,[g])$ in $a$, we have
\[
Y_{a_n,b}(M,\partial M,[g])\leq Y_{a_0,b}(M,\partial M,[g]).
\]
Since $u_n\in C_{a_n,b}(M)$, and under the assumption $R_g=0$, $H_g\geq 0$, we have
\[
\begin{split}
\frac{4(n-1)}{n-2}\int_{M}|\nabla u_n|_g^2dV_g&=E_M(u_n)-2(n-1)\int_{\partial M}H_gu_n^2d\sigma_g\\
&\leq Y_{a_n,b}(M,\partial M,[g])+\epsilon_1\\
&\leq Y_{a_0,b}(M,\partial M,[g])+\epsilon_1.
\end{split}
\]
This yields $\{u_n\}$ is uniformly bounded in $H^1(M,g)$.

Consider when $a_n\geq a$, then for sufficiently large $n$, we have
\[
1>a\int_{M}|u_n|^p+b\int_{\partial M}|u_n|^qd\sigma_g\geq 1-\epsilon_2.
\]
Again, we have a sequence $\{\lambda_n\geq 1\}$ such that $\lim_{n\to\infty}\lambda_n=1$, and $\lambda_nu_n\in C_{a,b}(M)$. 

For $n$ where $a_n< a$, just take $\lambda_n=1$.

By direct computation, we have
\[
\lambda_n\leq \frac{1}{(1-\epsilon_2)^{1/q}}.
\]
Consequently, we obtain
\[
Y_{a,b}(M,\partial M,[g])\leq E_M(\lambda_nu_n)=\lambda_n^2E_M(u_n)\leq \frac{Y_{a_n,b}(M,\partial M,[g])+\epsilon_1}{(1-\epsilon_2)^{1/q}}.
\]
Since $\epsilon_1$ can be arbitrarily small, and $\epsilon_2$ approaches $0$ as $n\to\infty$, we know
\[
Y_{a,b}(M,\partial M,[g])\leq \lim_{n\to\infty}Y_{a_n,b}(M,\partial M,[g]).
\]
Combining above, we have shown that $Y_{a,b}(M,\partial M,[g])$ is continuous in $a$. Similarly, one can show the continuity in $b$.
\end{proof}

As a direct corollary, we have
\begin{cor}
Suppose $(M,g)$ is a compact manifold with boundary, then $Y_{\lambda}(M,\partial M,[g])$ is continuous in $\lambda$ for any $\lambda\in [0,1]$.
\end{cor}


\section{Proof of the Main theorem}
We first need the continuity of $\sigma_{\lambda}(S_+^n)$ in $\lambda$.
\begin{lemma}
$\sigma_{\lambda}(S^n_+)$ is continuous in $\lambda$ for all $\lambda\in [0,1]$.
\end{lemma}
\begin{proof}
By Escobar's computation \cite{escobar1990uniqueness}, we know
\begin{displaymath}
\sigma_{\lambda}(S_+^n)=\frac{n(n-2)\text{Vol}_{g_0}(S_+^n)}{4(\lambda\text{Vol}_{g_0}(S_+^n)^{(n-2)/n}+(1-\lambda)\text{Area}_{g_0}(\partial S_+^n)^{(n-2)/(n-1)})},
\end{displaymath}
which is obviously continuous in $\lambda$.
\end{proof}
Now, we have all the ingredients we need for the Main Theorem.

\begin{proof}[proof of the Main Theorem]
\textit{We first consider the case when $\lambda\in(0,1]$. }

\vspace{0.5em}

By generalized Kobayashi's inequality, we have
\begin{displaymath}
\begin{split}
\sigma_{\lambda}(\#m_1\mathbb{RP}^n\#m_2&\mathbb{RP}^{n-1}\times S^1\# lH^n\# S_+^n)\geq \\
&\min\{\sigma_{\lambda}(\mathbb{RP}^n), \sigma_{\lambda}(\mathbb{RP}^{n-1}\times S^1),\sigma_{\lambda}(H^n)),\sigma_{\lambda}(S_+^n)\},
\end{split}
\end{displaymath}
and with the facts 
\begin{displaymath}
\sigma_{\lambda}(\mathbb{RP}^n),\sigma_{\lambda}(\mathbb{RP}^{n-1}\times S^1)\geq \sigma_{\lambda}(S_+^n),\quad \sigma_{\lambda}(H^n)=\sigma_{\lambda}(S_+^n),
\end{displaymath}
we could deduce that
\begin{displaymath}
\sigma_{\lambda}(\# m_1\mathbb{RP}^n\#m_2\mathbb{RP}^{n-1}\times S^1\#l H^n\# kS_+^n)\geq \sigma_{\lambda}(S_+^n),\quad \forall \lambda\in (0,1].
\end{displaymath}
Since $k+l\geq 1$, $\#m_1\mathbb{RP}^\#m_2\mathbb{RP}^{n-1}\times S^1\#lH^n\# kS_+^n$ is a manifold with boundary.
Escobar \cite{escobar1992conformal} showed that for any $n$-dimensional manifold with boundary and for any $\lambda\in[0,1]$, the $\lambda$-Yamabe invariant is less or equal to $\sigma_{\lambda}(S_+^n)$. 

Therefore, we conclude that
\[
\sigma_{\lambda}(\#m_1\mathbb{RP}^n\#m_2\mathbb{RP}^{n-1}\times S^1\#lH^n\# kS_+^n)=\sigma_{\lambda}(S_+^n),\quad \forall \lambda\in(0,1].
\]

\vspace{0.5em}

\textit{Now we consider the case when $\lambda=0$. }

\vspace{0.5em}

We denote $\#m_1\mathbb{RP}^\#m_2\mathbb{RP}^{n-1}\times S^1\#lH^n\# kS_+^n$ as $N$. 
Since we have shown that for any $\lambda\in (0,1]$, $\sigma_{\lambda}(N)=\sigma_{\lambda}(S_+^n)$, for any fixed $\lambda\in(0,1]$.Then for arbitrarily small $\epsilon>0$, there exists a conformal class of Riemannian metric $[g]$ on $N$, such that
\begin{displaymath}
Y_{\lambda}(N,\partial N, [g])\geq \sigma_{\lambda}(N)-\epsilon=\sigma_{\lambda}(S_+^n)-\epsilon.
\end{displaymath}
By continuity of $Y_{\lambda}(N,\partial N, [g])$ in $\lambda$, we know
\[
Y_0(N,\partial N, [g])=\lim_{\lambda\to 0}Y_{\lambda}(N,\partial N,[g])=\sigma_0(S_+^n)-\epsilon.
\]

With the fact that $\epsilon$ can be arbitrarily small, we conclude that
\begin{displaymath}
\sigma_{0}(N)\geq \sigma_0(S_+^n).
\end{displaymath}
Notice that, the above equation holds for any $\lambda\in (0,1)$, and we have shown $\sigma_{\lambda}(S_+^n)$ is continuous in $\lambda$ for any $\lambda\in [0,1]$, we can take $\lambda\to 0$, and then conclude
\begin{equation}
\sigma_0(N)\geq \sigma_0(S_+^n).
\end{equation}
By \ref{Escobarupperbound}, we conclude
\begin{equation}
\sigma_0(N)=\sigma_0(S_+^n).
\end{equation}
Combining the above two cases, we have completed the proof of the Main Theorem.
\end{proof}

\section{Other Examples and Applications}
The main idea of our proof lies in the observation that, if the Yamabe invariant of a closed $n$ manifold is greater than or equal to that of $\sigma_1(S_+^n)$, then connect sum of this manifold with $S_+^n$ or $H^n$ has the same $\lambda$-Yamabe invariant as $S_+^n$ for any $\lambda\in [0,1]$.

In dimension 4, we have the following examples.
\begin{ex}[Lebrun \cite{lebrun1996yamabe}]\label{lebrunexample}
\[
\sigma_1(\mathbb{CP}^2)=12\sqrt{2}\pi\geq \sigma_1(S_+^4).
\]
\end{ex}

With the generalized Kobayashi's inequality, and run the above arguments again, in dimension 4, we have
\begin{equation}
\sigma_{\lambda}(\#m_1\mathbb{RP}^4\#m_2\mathbb{RP}^{4}\times S^1\# m_3 \mathbb{CP}^2\#lH^4\#k S_+^4)=\sigma_{\lambda}(S_+^4),\quad \forall \lambda\in [0,1],
\end{equation}
provided $k+l\geq 1$.

\begin{ex}[Peaten and Ruiz \cite{petean2011isoperimetric}]
\begin{displaymath}
\sigma_1(S^2\times S^2)\geq \sigma_1(S_+^4).
\end{displaymath}
\end{ex}
With the generalized Kobayashi's inequality, and run the above arguments again, in dimension 4, we have
\begin{equation}
\sigma_{\lambda}(\#m_1\mathbb{RP}^4\#m_2\mathbb{RP}^{n-1}\times S^1\# m_3\mathbb{CP}^2\#m_4 S^2\times S^2\#lH^4\#k S_+^4)=\sigma_{\lambda}(S_+^4),\quad \forall \lambda\in [0,1],
\end{equation}
provided $k+l\geq 1$.

Another significant theorem proved by Petean \cite{petean2000yamabe} is
\begin{thm}[Petean]
Every simply connected smooth closed manifold of dimension greater than four has non-negative Yamabe invariant.
\end{thm}

With the Generalized Kobayshi's inequality, we can deduce
\begin{cor}
Suppose $M^n$ is any simply connected closed manifold with $n\geq 5$, then
\begin{equation}
\sigma_{\lambda}(\# mM^n\#m_1\mathbb{RP}^{n}\#m_2\mathbb{RP}^{n-1}\times S^1\# lH^n\#k S^n_+)\geq 0,\quad \forall \lambda\in [0,1].
\end{equation}
\end{cor}
\begin{proof}
When $\lambda\in(0,1]$, this is a direct corollary. 

Now we consider the case when $\lambda=0$, and denote
\begin{displaymath}
N=\# mM^n\#m_1\mathbb{RP}^{n}\#m_2\mathbb{RP}^{n-1}\times S^1\# lH^n\#k S^n_+.
\end{displaymath}
For any fixed $\lambda\in(0,1)$, and any $\epsilon>0$, there exists a conformal class of Riemannian metric $[g_{\epsilon}]$ on $N$, such that
\begin{displaymath}
Y_{\lambda}(N,\partial N, [g_{\epsilon}])\geq \sigma_{\lambda}(N)-\epsilon\geq -\epsilon.
\end{displaymath}
Then for any $u\in C_{0,1}(N)$, we have
\[
\lambda\int_{N}|u|^pdV_{g_{\epsilon}}+(1-\lambda)\int_{\partial N}|u|^qd\sigma_{g_{\epsilon}}\geq 1-\lambda.
\]
Consequently, we have
\[
\lambda\int_{N}|Cu|^pdV_{g_{\epsilon}}+(1-\lambda)\int_{\partial N}|Cu|^qd\sigma_{g_{\epsilon}}\geq 1.
\]
where $C=(1-\lambda)^{-1/q}$.
Then we can obtain that, for any $u\in C_{0,1}(N)$
\[
E_N(u)=\frac{E_N(Cu)}{C^2}\geq \frac{1}{(1-\lambda)^{2/q}}Y_{\lambda}(N,\partial N, [g_{\epsilon}])\geq -\epsilon,
\]
which implies
\[
\sigma_0(N)\geq Y_0(N,\partial N, [g_\epsilon])\geq -\epsilon.
\]
Since $\epsilon$ can be arbitrarily small, we know
\begin{displaymath}
\sigma_0(N)\geq 0.
\end{displaymath}
\end{proof}

\bibliographystyle{plain} 
\bibliography{ref} 

\end{document}